%% file: control-master.tex
\documentclass[12pt]{amsart}

\usepackage{amsmath}
\usepackage{amssymb}
\usepackage{amsthm}
\usepackage{bbm}
\usepackage{hyperref}
\usepackage{xypic}

\usepackage{fullpage}

\usepackage{xcolor}
\hypersetup{
  colorlinks=true,
  allcolors=green!50!black
}

\theoremstyle{plain}
\newtheorem{theorem}{Theorem}[section]
\newtheorem{corollary}[theorem]{Corollary}
\newtheorem{definition}[theorem]{Definition}
\newtheorem{proposition}[theorem]{Proposition}
\theoremstyle{remark}
\newtheorem{remark}[theorem]{Remark}

\numberwithin{equation}{section}

\DeclareMathOperator{\diag}{diag}
\DeclareMathOperator{\id}{id}

\DeclareMathOperator{\pr}{pr}
\def\shf{\mathcal}

\newcommand{\uparrowx}{\operatorname{\uparrow}x}

\newcommand{\wtCv}{\widetilde{C_v}}
\newcommand{\wtcv}{\widetilde{c_v}}
\newcommand{\wtFv}{\widetilde{F_v}}
\newcommand{\wtfv}{\widetilde{f_v}}
\newcommand{\wtRv}{\widetilde{R_v}}
\newcommand{\wtSv}{\widetilde{S_v}}
\newcommand{\wtsv}{\widetilde{s_v}}
\newcommand{\wtsw}{\widetilde{s_w}}

\newcommand{\wtshfL}{\widetilde{\shf{L}}}

\newcommand{\wtshfN}{\widetilde{\shf{N}}}
\newcommand{\wtshfNn}{\widetilde{\shf{N}_n}}
\newcommand{\wtshfS}{\widetilde{\shf{S}}}

\newcommand{\bv}[1]{\mathbf{#1}}
\newcommand{\veceta}{\boldsymbol{\eta}}
\newcommand{\vech}{\bv{h}}
\newcommand{\vechs}{\bv{h_s}}
\newcommand{\vecone}{\mathbbm{1}}
\newcommand{\vecu}{\bv{u}}
\newcommand{\vecx}{\bv{x}}
\newcommand{\vecy}{\bv{y}}
\newcommand{\wtvecu}{\widetilde{\vecu}}
\newcommand{\wtvecx}{\widetilde{\vecx}}

\begin{document}

% --------------------------------------------------------------------------------------------------
% Front-matter setup

\title{Sheaf-theoretic framework for optimal network control}

\author{Griffin M. Kearney$^*$}\thanks{This material is based upon work supported in part by the Air Force Research Laboratory (AFRL)
under award number FA8750-18-C-0016.
The views and conclusions contained herein are those of the authors
and should not be interpreted as necessarily representing the official policies or endorsements,
either expressed or implied,
of the United States Air Force.
Cleared \mbox{AFRL-2020-0378} 27 Nov 2020. \\ \indent
$^*$SRC, Inc., North Syracuse, NY 13212 (\{gkearney,kpalmowski,mrobinson\}@srcinc.com)}
\author{Kevin F. Palmowski$^*$}
\author{Michael Robinson$^{*\dagger}$}\thanks{$^\dagger$Department of Mathematics and Statistics, American University, Washington, DC 20016 (michaelr@american.edu)}

% Required to override amsart default ABSTRACT format.
\renewcommand{\abstractname}{\textbf{Abstract}}

\begin{abstract}
In this paper,
we use tools from sheaf theory to model and analyze optimal network control problems and their associated discrete relaxations.
We consider a general problem setting in which pieces of equipment and their causal relations are represented as a directed network,
and the state of this equipment evolves over time according to known dynamics and the presence or absence of control actions.
First, we provide a brief introduction to key concepts in the theory of sheaves on partial orders.
This foundation is used to construct a series of sheaves that build upon each other to model the problem of optimal control,
culminating in a result that proves that solving our optimal control problem is equivalent to
finding an assignment to a sheaf that has minimum consistency radius and restricts to a global section on a particular subsheaf.
The framework thus built is applied to the specific case where a model is discretized to one in which
the state and control variables are Boolean in nature,
and we provide a general bound for the error incurred by such a discretization process.
We conclude by presenting an application of these theoretical tools that demonstrates
that this bound is improved when the system dynamics are affine.

\vspace{0.1in}

\vspace{0.1in}
\end{abstract}

% --------------------------------------------------------------------------------------------------
%	Make the title

\maketitle

% --------------------------------------------------------------------------------------------------
% Sections

% Introduction
\input{intro.tex}

% Basics of sheaves and topology
\input{basics.tex}

% Problem statement
\input{problem-statement.tex}

% Encoding as a sheaf
\input{sheaf-encoding.tex}

% Relaxing to BTM
\input{btm-relaxation.tex}

% Application to discrete initial control
\input{application-boolean-control.tex}

% Conclusion
\input{conclusion.tex}

% --------------------------------------------------------------------------------------------------
% Bibliography

\bibliographystyle{plain}
\bibliography{control-master}

\end{document}

%% file: intro.tex
% !TeX root=./control-master.tex

\section{Introduction}
\label{sec-intro}

Optimization is a useful framework for improving the design and control of network systems.
Applications of optimization techniques to network problems arise in resource allocation,
increased system robustness, and increased efficiency within transport networks
\cite{roscomfag19,keafar18,com17,farkea17,savcomdah14,comsavacedahfra13a,comsavacedahfra13b,comsavacedahfra12,zil00},
as well as in emerging areas like viral marketing
\cite{claleealobuspoo18,aceozdtah16,batdelgalgresti12,easkle10,boyvan04}.
Because of the high computational complexity of solving optimization problems in these practical settings,
there are numerous techniques for simplifying these problems through approximations and relaxations.
One such approximation strategy relies on coarsening the (nominally continuous) states of the network into discrete values.
When the network infrastructure admits a Boolean state description,
state discretizations can provide substantial computational savings.
The practical value of a given state discretization is greatly dependent on the amount of error induced by the coarsening process:
a discretization that induces too much error is generally useless,
so it is of the utmost importance to characterize and quantify the nature of this discretization error.
In this paper,
we develop a theoretical framework for evaluating these errors in a general setting.

% What's the takeaway? (Plain English, minimal jargon)

The main result of this paper is Theorem \ref{sec-btm-relaxation_subsec-lifting-thresh-thm},
which provides an explicit bound on the overall amount of error incurred by
approximating an optimization problem on a network
with a discretized network optimization problem.
The precise specification of the bound is rather intricate,
as it is quite general in the sense that it applies to \emph{any} network optimization problem.
With this bound in hand,
we apply it to the case of discretization to network states represented as vectors of Boolean values.
Since efficient optimization solvers for integer programs exist
\cite{boyvan04},
solving the discretized problem is much easier than solving the original continuous problem.
The bound from the theorem estimates how much error this process will incur.

% What are the main players?

Since Theorem \ref{sec-btm-relaxation_subsec-lifting-thresh-thm} is extremely general,
it relies upon a mathematical framework that is both sufficiently expressive to capture all optimization problems on networks,
whether continuous or discrete,
and sufficiently refined so as to provide analytic bounds.
These two competing requirements are met by the mathematical theory of \emph{sheaves}.

Sheaves are the foundational mathematical theory for modeling local consistency relationships among data.
We use them to express local consistency between state variables on neighboring portions of the network
and between neighboring time steps.
As a necessary first step,
the dynamical model of the network, its discretization,
and the cost function to be optimized must all be encoded within the context of sheaves.
Theorem \ref{thm:dynamic_optimal} establishes that all network optimization problems can be encoded sheaf-theoretically in a lossless,
structure-preserving way.

% How does this fit into the literature? Network Opt

In practice,
optimal control formulations are often numerically difficult to solve.
The development of approximations and relaxations for intractable control problems provides practical improvements at the expense of accuracy.
Therefore,
it is of great importance to quantify how much an approximate solution deviates from the exact solution.
While related work has been carried out in order to develop error bounds for specific problems
\cite{comsavacedahfra13a,comsavacedahfra13b,easkle10,boyvan04},
the sheaf-based methods and results in this paper are generally applicable to wide classes of optimal network control problems.

% How does this fit into the literature? TDA

Sheaf theory is based upon the theory of \emph{topological spaces},
and the application of sheaves to various problems fits into the growing field of \emph{topological data analysis}.
The topology of the network is an integral part of our sheaf encoding,
but plays an implicit role in our results,
so the topology of the network can impact the overall error bound.
It is well established that topological properties of dynamical systems play a large role in their behavior
\cite{guckenheimer2013nonlinear,bush2012combinatorial,mischaikow2002topological}.
Moreover, there has been recent interest in the topological properties of time series generated by a dynamical system
\cite{perea2019topological,xu2019twisty,tralie2018quasi,perea2015sliding,mischaikow1999construction}.
These works focus primarily on discrete observations of a continuous dynamical system,
while this paper is instead focused on the incurred discretization error.

% How does this fit into the literature? applied sheavery

The generality of sheaves makes them an increasingly popular tool for analyzing complex systems.
The encodings herein are based upon the general recipe expressed in
\cite{robinson2017sheaf}.
In that work, discretization error was found to be expressed as the failure of certain sheaf diagrams to be commutative.
This failure can be expressed as a consistency radius, which was defined in
\cite{robinson2017sheaves}
and subsequently expanded in
\cite{robinson2019hunting,robinson2018assignments}.
That said,
the expression of a network dynamical system,
constructed in this article using a \emph{sheaf morphism} to encode the cost function for optimization,
also appears to be novel.

% --------------------------------------------------------------------------------------------------

\subsection{Organization of the paper}

We begin with a brief introduction to topology and sheaf theory in Section \ref{sec-basic_sheaves},
which provides a self-contained treatment of the concepts needed for the rest of the article.
The reader who is already familiar with sheaves may find it expedient to skim Section \ref{sec-basic_sheaves} for notation,
referring back to it as necessary.
Preliminary definitions and a statement of the problem of interest are the subject of Section \ref{sec-problem-statement}.
In Section \ref{sec-sheaf-encoding},
we construct a series of sheaves that, taken together, encode the problem of interest.
This sheaf-theoretic framework is used to encode a Boolean relaxation of our original problem in Section \ref{sec-btm-relaxation},
where Theorem \ref{sec-btm-relaxation_subsec-lifting-thresh-thm}
provides bounds on the approximation error incurred by this discretization.
As noted, the theorem applies quite generally,
and Section \ref{sec-appl-boolean-control} provides a detailed example
that demonstrates how the error bounds can be improved
with additional problem-specific knowledge and assumptions,
adding practical relevance to the theoretical result.

%% file: basics.tex
% !TeX root=./control-master.tex

% --------------------------------------------------------------------------------------------------

\section{A brief introduction to topology and sheaves}
\label{sec-basic_sheaves}

We refer the reader to several recent introductions into the use of sheaves in data analysis \cite{robinson2019hunting,robinson2017sheaves,robinson2014topological},
but recount the necessary basics here.
A \emph{sheaf} is a mathematical object that encodes consistency relations among data associated to the elements of a partially ordered set.
Intuitively, a sheaf labels elements of a partially ordered set with various sets of possible values,
and these values are related to one another by functions that are
defined whenever the underlying elements of the partially ordered set are related.
Specifically, let $P$ be a set and $\le_P$ be a \emph{partial order} on the elements of $P$, which satisfies
\begin{itemize}
\item \textbf{Reflexivity:} $x \le_P x$ for all $x\in P$,
\item \textbf{Antisymmetry:} At most one of $x \le_P y$ or $y \le_P x$ is true if $x \not= y$, and
\item \textbf{Transitivity:} $x \le_P z$ whenever $x \le_P y$ and $y \le_P z$.
\end{itemize}
Elements $x, y \in P$ are \emph{comparable} if $x \le_P y$ or $y \le_P x$, and \emph{incomparable} otherwise.
A partial order is said to be \emph{locally finite} if for each $x \in P$ there are only finitely many $y \in P$ with $x \le_P y$.
Typically, we represent a locally finite partial order $(P,\le_P)$ by its \emph{Hasse diagram},
a directed graph with a vertex for each element of $P$.
Each edge of the Hasse diagram $x \to y$ implies that $x \le_P y$,
and the edge set of a Hasse diagram is minimal in the sense that all relations
$x \le_P y$ that do not have a corresponding $x \to y$ can be recovered by transitivity
of the relations that \emph{are} represented in the diagram.
We typically write $\le_P$ as $\le$ to reduce visual clutter if there is no opportunity for confusion.

It is a useful fact that every partially ordered set $(P,\le)$ can be endowed with the \emph{Alexandrov} topology,
a canonical topology that is generated by sets of the form
\[
\uparrowx := \{ y \in P : x \le y \}
\]
for $x\in P$.
Arbitrary unions of these $\uparrowx$ form the \emph{open sets} in the Alexandrov topology.

% ------------------------------------------------
%	DEF

\begin{definition}
A \emph{sheaf $\shf{S}$ on the partially ordered set $(P,\le)$} consists of the following specification:
\begin{itemize}
\item \textbf{\emph{Stalks:}} Sets $\shf{S}(x)$ for each $x$ in $P$, and
\item \textbf{\emph{Restrictions:}} Functions $\shf{S}(x \le y) \colon \shf{S}(x) \to \shf{S}(y)$ for each related pair of elements $x \le y$ in $P$,
such that $\shf{S}(x \le z) = \shf{S}(y \le z) \circ \shf{S}(x \le y)$ whenever $x \le y \le z$ in $P$.
\end{itemize}
\end{definition}
The notation for $\shf{S}$ is polymorphic: $\shf{S}(x)$ is a set, while $\shf{S}(x \le y)$ is a function between sets.

%	DEF
% ------------------------------------------------

We will usually represent a sheaf as an annotated Hasse diagram,
in which the vertices are labeled with their stalks and the edges are labeled with their restrictions.
As a simple example,
consider the partially ordered set $\{a,b,c,d\}$ in which the partial order is generated by
$a \le b$, $a \le c$, $b \le d$, and $c \le d$, but $b$ and $c$ are incomparable.
The Hasse diagram for this scenario is shown at left below.
\begin{equation} \label{eqn:ex-shf-diagram}
\xymatrix @R = 0.5pc {
&d&&&\shf{S}(d)&\\
b \ar[ur] & & c \ar[ul]&\shf{S}(b) \ar[ur]^-{\shf{S}(b\le d)} & & \shf{S}(c) \ar[ul]_-{\shf{S}(c \le d)}\\
& a \ar[ur]\ar[ul] &&& \shf{S}(a) \ar[ur]_-{\shf{S}(a\le c)}\ar[ul]^-{\shf{S}(a \le b)} &
}
\end{equation}
Any sheaf on this partially ordered set can be described using the same diagram.
For instance,
if the sheaf is named $\shf{S}$,
then its diagram is shown at right in \eqref{eqn:ex-shf-diagram}.
This diagram shows all of the stalks and almost all of the restrictions.
The remaining restriction, $\shf{S}(a \le d)$, can be reconstructed via transitivity, since
\[\shf{S}(a \le d)  = \shf{S}(b \le d) \circ \shf{S}(a \le b)  = \shf{S}(c \le d) \circ \shf{S}(a \le c). \]

The diagram \eqref{eqn:ex-shf-diagram} is an example of a \emph{commutative} diagram,
which is a directed graph in which the edges are labeled with functions
such that composition of functions does not depend on the path taken within the graph.
The transitivity axiom guarantees that the diagrams for sheaves on partially ordered sets are always commutative.

% ------------------------------------------------
%	EG

In order to encode an optimization problem as a sheaf,
it is useful to recognize that some (rather trivial) sheaves are easy to define.
For instance, the \emph{trivial sheaf} on an arbitrary partial order $(P,\le)$
is written $\widehat{0}$.
Each stalk of $\widehat{0}$ is the trivial vector space $0$,
containing only the zero element.
This specification completely determines all of the restrictions to be zero maps.

%	EG
% ------------------------------------------------

Continuing with the interpretation of a sheaf $\shf{S}$ on $(P,\le)$ as specifying a way to label elements of $P$,
a \emph{global assignment} of $\shf{S}$ consists of an element of the product $\prod_{x \in P} \shf{S}(x)$,
while a \emph{local assignment} is an element of a similar product over any subset of $P$.
We use $a_x$ to denote the value that an assignment $a$ assigns to $x \in P$.
Any assignment $s$ that agrees with the restrictions,
which is to say that it satisfies
$s_y = \left( \shf{S}(x \le y) \right)\left(s_x\right)$
for all $x \le y$ where $s$ is defined, is called a \emph{section}%
\footnote{The reader who happens to be familiar with the traditional literature on sheaves may notice that our definition of the stalk of a sheaf is much simpler than usual, since stalks typically involve a limit construction. Such a limit simply reduces to the space of sections over a single open set of the form $\uparrowx$, as our definition demands.}
of $\shf{S}$.
Like assignments, sections can be local or global.

In what follows, the stalks will be taken to be sets of state variables.
In order to be useful quantitatively, these sets should be endowed with a \emph{pseudometric},
which allows distances between states to be computed.
Specifically, a \emph{pseudometric} $d$ on a set $X$ is a function
$d \colon X \times X \to \mathbb{R}$ that satisfies the following properties for all $x, y, z \in X$:
\begin{itemize}
\item \textbf{Symmetry:} $d(x,y) = d(y,x)$,
\item \textbf{Reflexivity:} $d(x,x) = 0$,
\item \textbf{Nonnegativity:} $d(x,y) \ge 0$, and
\item \textbf{Triangle inequality:} $d(x,z) \le d(x,y) + d(y,z)$.
\end{itemize}

When the domain and codomain of a function $f \colon X \to Y$ both have pseudometrics,
$d_X$ and $d_Y$,
respectively,
$f$ is said to be \emph{Lipschitz} if there is a $K>0$ such that
$d_Y\left(f(x_1),f(x_2)\right) \le K d_X(x_1,x_2)$
for all $x_1, x_2 \in X$.
We call the infimum of all such $K$ the \emph{Lipschitz constant} for $f$.

We will assume that all stalks of all sheaves in this article are endowed with (possibly different) pseudometrics and
that each restriction is a Lipschitz function.
In this case, we say that $\shf{S}$ is a sheaf \emph{of pseudometric spaces}.
Under this assumption,
we can estimate how far a given assignment is from being a section using its \emph{consistency radius}.

% ------------------------------------------------
%	DEF

\begin{definition}
If $a$ is an assignment to a sheaf $\shf{S}$ of pseudometric spaces on a partially ordered set $(P,\le)$,
then the quantity
\[ c_{\shf{S}}(a) := \sqrt{\sum_{\{(x,y) \in P \times P \, : \, x \le y\}} d_y ( a_y, ( \shf{S}(x \le y) ) ( a_x ) )^2}, \]
where $d_y$ is the pseudometric on the stalk $\shf{S}(y)$,
is called the \emph{consistency radius of $a$}.
\end{definition}

%	DEF
% ------------------------------------------------

There are several useful results pertaining to the consistency radius.
The most basic is that it is a lower bound on the distance between an assignment and the nearest global section.
To understand that bound,
it is useful to use the \emph{assignment pseudometric}, given by
\[d_{\shf{S}} (a,b) := \sqrt{\sum_{x \in P} d_x(a_x,b_x)^2}\]
for two assignments $a$ and $b$ to a sheaf $\shf{S}$.
Intuitively, the assignment pseudometric measures the distance between two assignments.

% ------------------------------------------------
%	PROP

\begin{proposition}\cite[Prop. 23]{robinson2017sheaves}
\label{sec-sheaf-encoding_prop-crbound}
Suppose that $a$ is an assignment to a sheaf $\shf{S}$ of pseudometric spaces on a locally finite partially ordered set.
If every restriction map of $\shf{S}$ has Lipschitz constant less than or equal to $K$, then
\[c_{\shf{S}}(a) \le (1+K) d_{\shf{S}}(s,a)\]
for every global section section $s$ of $\shf{S}$.
\end{proposition}

%	PROP
% ------------------------------------------------

If $(P, \le_P)$ and $(Q, \le_Q)$ are partially ordered sets,
then we say that the function $f \colon P \to Q$ is \emph{order preserving}
if $f(x) \le_Q f(y)$ whenever $x \le_P y$.
Order preserving functions can be used to transform sheaves and assignments in a natural way.

% ------------------------------------------------
%	DEF

\begin{definition}
Suppose that $\shf{S}$ is a sheaf on $(P,\le_P)$,
$\shf{R}$ is a sheaf on $(Q,\le_Q)$,
and $f \colon P \to Q$ is an order preserving function.
A \emph{sheaf morphism $m \colon \shf{R} \to \shf{S}$ along $f$} consists of the specification of
\emph{component functions} $m_x \colon \shf{R}(f(x)) \to \shf{S}(x)$ for each $x \in P$ such that
\[ \shf{S}(x \le_P y) \circ m_x = m_y \circ \shf{R}(f(x)\le_Q f(y))  \]
for all $y \in P$ with $x \le_P y$.
\end{definition}

%	DEF
% ------------------------------------------------

A sheaf morphism $m \colon \shf{R} \to \shf{S}$ transforms the assignments of $\shf{R}$ into assignments of $\shf{S}$,
simply by applying the component functions stalk-wise.
Explicitly, if $r$ is an assignment to $\shf{R}$,
then $m(r)$ is an assignment to $\shf{S}$ whose value on $x$ is given by
\[(m(r))_x := m_x(r(f(x))).\]

The most useful result about sheaf morphisms is that they transform the consistency radius of assignments in a controlled way,
as demonstrated by the following bound,
which is a generalization of \cite[Lemma 4]{robinson2018assignments}.

% ------------------------------------------------
%	PROP

\begin{proposition}
\label{sec-sheaf-encoding_prop-morphism}
Let $\shf{R}$ and $\shf{S}$ be sheaves of pseudometric spaces on the partially ordered sets
$(Q, \le_Q)$ and $(P, \le_P)$, respectively.
Suppose that $f \colon P \to Q$ is an order preserving function
and that each $x \in P$ has a corresponding Lipschitz function $m_x \colon \shf{R}\left( f(x) \right) \to \shf{S}(x)$.
Given any assignment $a$ of $\shf{R}$,
construct an assignment $b$ of $\shf{S}$ by setting $b_x := m_x( a_{f(x)})$ for each $x \in P$.
If there exists some $\epsilon \geq 0$ such that
\[ d_{\shf{S}(y)} \left( \left( \shf{S}(x \le_P y) \circ m_x \right)(z), \, \left( m_y \circ \shf{R}(f(x) \le_Q f(y)) \right)(z) \right) \leq \epsilon \]
for all $x, y \in P$ with $x \leq_P y$ and for all $z \in \shf{R}(f(x))$, then
\[ c_{\shf{S}}(b) \le K c_\shf{R}(a) + C \epsilon, \]
where $K > 0$ is any upper bound on the Lipschitz constants of all of the $m_x$ maps,
and $C^2$ is the total number of restrictions present in $\shf{S}$.
\end{proposition}

As a consequence, sheaf morphisms preserve global sections, since for sheaf morphisms, $\epsilon=0$.

\begin{proof}
Let $C^2 := | \{(x,y) \in P \times P : x \le_P y\} |$, the number of restrictions in $\shf{S}$.
By definition,
\begin{align*}
c_{\shf{S}}(b)
	&= \sqrt{\sum_{\{(x,y) \in P \times P \, : \, x \le_P y\}} \left[d_{\shf{S}(y)}\left(b_y,\shf{S}(x \le_P y)\left(b_x\right)\right)\right]^2}\\
	&= \sqrt{\sum_{\{(x,y) \in P \times P \, : \, x \le_P y\}} \left[d_{\shf{S}(y)}\left(m_y\left(a_{f(y)}\right), \shf{S}(x \le_P y)\left(m_x\left(a_{f(x)}\right)\right)\right)\right]^2}\\
	&\le \sqrt{\sum_{\{(x,y) \in P \times P \, : \, x \le_P y\}} \left[d_{\shf{S}(y)}\left(m_y\left(a_{f(y)}\right), m_y\left( \shf{R}(f(x) \le_Q f(y))\left(a_{f(x)}\right)\right)\right) + \epsilon\right]^2}\\
	&\le \sqrt{\sum_{\{(x,y) \in P \times P \, : \, x \le_P y\}} \left[K d_{\shf{R}(f(y))}\left(a_{f(y)}, \shf{R}(f(x) \le_Q f(y))\left(a_{f(x)}\right)\right) \right]^2} + C \epsilon\\
%	&\le K \sqrt{\sum_{\{(x,y) \in P \times P \, : \, x \le_P y\}} \left[d_{\shf{R}(f(y))}\left(a_{f(y)}, \shf{R}(f(x) \le_Q f(y))\left(a_{f(x)}\right)\right)\right]^2} + C \epsilon\\
	&\le K \sqrt{\sum_{\{(u,v) \in Q \times Q \, : \, u \le_Q v\}} \left[d_{\shf{R}(v)}\left(a_{v}, \shf{R}(u \le_Q v)\left(a_{v}\right)\right)\right]^2} + C \epsilon\\
	&\le K c_{\shf{R}}(a) + C \epsilon. \qedhere
\end{align*}
\end{proof}

%	PROP
% ------------------------------------------------

%% file: problem-statement.tex
% !TeX root=./control-master.tex

\section{Problem statement}
\label{sec-problem-statement}

% --------------------------------------------------------------------------------------------------

\subsection{Underlying dynamical system}
\label{sec-problem-statement_subsec-dynamical-system}

Let $G := (V,E)$ be a finite directed graph in which vertices $V$ represent controllable \emph{pieces of equipment}
(whether physical or virtual)
and edges $E$ represent causal relationships between pieces of equipment.
We assume that for every vertex $v \in V$, there is a self-edge $(v,v) \in E$.
It is useful to define the \emph{$1$-hop neighborhood} of each vertex $v$ by
\[
U_v := \{ w \in V : \text{there exists a } (w,v) \in E \}.
\]
The presence of self-edges implies that $v \in U_v$ for all vertices $v$.
The intuition is that the $1$-hop neighborhood $U_v$ lists all of the vertices whose states impact the state of $v$.

Each vertex $v$ is labeled with a space of \emph{endogenous state variables} $S_v$ and a space of \emph{exogenous control variables} $C_v$.
(This framework includes autonomous systems as well -- one merely needs to supply extra state variables to $S_v$ to represent the state of the internal control system of a given piece of equipment.)
The state variables $S_v$ evolve deterministically according to a fixed set of causal rules,
but we interpret the values of $C_v$ as being determined outside the system,
so there is no corresponding assumption on their behavior.
In order to apply our methods, we assume that the $C_v$ and $S_v$ spaces
are arbitrary \emph{pseudometric spaces},
equipped with compatible \emph{Borel measures}.

Each piece of equipment's state is governed by a discrete-time dynamical system,
which is formalized as a continuous function from its state variables
and those of its neighbors
to its state variables.
For simplicity, we will assume that time is synchronized across all pieces of equipment,
which means that they all share the same time steps.
(This is not strictly necessary for our results to hold,
but greatly simplifies the analysis.)
The collection of variables at a vertex $v$ that could be involved in governing the state of $v$ at future times is denoted by
$R_v := C_v \times \prod_{w \in U_v} S_w$.
This formulation assumes that the control variables are only visible to vertex $v$ at the current time;
later times may implicitly see the effect of control variables in other vertices from previous times.
If it should be the case that we want to represent control variables that are visible to many vertices,
then our formulation can support that by way of replacing a set of vertices with a single vertex with a state aggregated across the original vertices.

Dynamics are represented by Lipschitz functions $f_v \colon R_v \to S_v$ local to each vertex $v\in V$.
Since it will be important when considering discretizations,
we note that Lipschitz functions are automatically Borel measurable functions.

A state for the system that follows all applicable physical laws and other system requirements is called \emph{feasible}.
We use $F_v \subseteq R_v$ to denote the space of feasible system states.
As a minor constraint,
we require that the evolution of a feasible state under the dynamics $f_v$ remains feasible.
Mathematically, $F_v$ must be an invariant set for $f_v$:
$f_v(F_v) \subseteq \pr_{S_v} F_v$,
where $\pr_A \colon A \times B \to A$ represents the projection of a product onto the listed factor.
Note that invariance does not preclude later infeasibility due to inappropriate control actions being taken later in time.

% --------------------------------------------------------------------------------------------------

\subsection{Objective function}
\label{sec-problem-statement_subsec-objective-fnc}

The goal of the operator of the network system is to arrange for the system's state to meet certain requirements.
For each vertex $v$,
we define \emph{objective functions}
$J_v \colon S_v \to \mathbb{R}$ and $J'_v \colon R_v \to \mathbb{R}$.
Each $J_v$ maps the state of $v$ to a nonnegative real number that can be interpreted as a penalty for having $v$ in that particular state.
Since the domain of $J'_v$ additionally contains control variables,
the value of $J'_v$ can be interpreted as supplying a cost for choosing certain control actions.
Although it is reasonable to assume that $J_v$ and $J'_v$ agree about the penalties for $v$ being in a certain state,
we need not require this.
The operator's goal is therefore represented as a minimization -- ideally to zero -- of each of these $J_v$ and $J'_v$ functions across all vertices in the network.
To better manage noise in the system,
we will minimize a global objective function given by the sum of squares of the $J_v$ and $J'_v$ functions.

%% file: sheaf-encoding.tex
% !TeX root=./control-master.tex

\section{Encoding as a sheaf}
\label{sec-sheaf-encoding}

This section explains how to construct sheaves that encode the problem setup described in Section \ref{sec-problem-statement}.
In these sheaves, the base space topology comes from the network topology in a rather direct way:
it is the Alexandrov topology on the \emph{face partial order} for the graph.
The dynamical structure of the model is encoded in the \emph{restriction maps} of the sheaves.
This section constructs five sheaves:
\begin{enumerate}
\item $\shf{N}$, whose global sections correspond to the state of the network at a single time step (Proposition \ref{prop:net_state_sheaf}),
\item $\shf{M}$, whose global sections correspond to solutions to the optimal control problem for a time-invariant system (Proposition \ref{prop:static_optimal}),
\item $\shf{L}$, whose global sections propagate state from one time step to the next,
\item $\shf{T}$, whose global sections correspond to feasible trajectories of the network through time (Proposition \ref{prop:dynamic_sheaf}), and
\item $\shf{S}$, whose global sections correspond to solutions to the optimal control problem (Theorem \ref{thm:dynamic_optimal}).
\end{enumerate}

% --------------------------------------------------------------------------------------------------

\subsection{Single time step as a sheaf}
\label{sec-sheaf-encoding_subsec-single-time-step}

It is useful to render the graph $G = (V,E)$ described in Section \ref{sec-problem-statement} as a partially ordered set $(P,\le)$,
in which $P=V \cup \left\{ U_v : v \in V \right\}$.
There are two classes of elements in $P$:
(1) vertices $v$, and
(2) $1$-hop neighborhoods $U_v$ of vertices.
The order relation on $P$ is then generated by all relations of the form $U_v \le w$ if $w\in U_v$.

This partial order is ranked with two levels:
the $1$-hop neighborhoods are on the lower level,
and the vertices are on the upper level.
As a result of this structure,
the transitivity axiom for a partial order is trivial,
since the only way that a chain of relations like
$x \le y \le z$
can occur is if $x = y$ or $y = z$.
Consequently, the transitivity axiom for sheaves is trivially satisfied by any choice of stalks and restrictions on this partial order,
so any choice of restriction functions that agree about their domains whenever they have a $U_v$ in common will define a sheaf.

We proceed to define a sheaf $\shf{N}$ that represents the state of the network at a given time.
The sheaf $\shf{N}$ assigns the stalk $\shf{N}(v) := S_v$,
the set of endogenous state variables, to each vertex $v$,
and assigns the stalk $\shf{N}(U_v):=F_v$,
the set of feasible state and control variables, to the $1$-hop neighborhood of $v$.
The restriction maps are given by projections onto the state variables,
namely $\shf{N}(U_v \le w) := \pr_{S_w}$,
a function $F_v \to S_w$ whenever $w\in U_v$.
This is well-defined since $F_v \subseteq R_v$,
and $R_v$ is a product of the control variables $C_v$ and the state variables $S_w$ of every edge $(w,v) \in E$ incident to $v$.

With this construction,
a typical sheaf diagram for $\shf{N}$ will look like
\begin{equation} \label{eqn:shf-N}
\xymatrix @R = 1pc @C = 2.5pc {
\dotsb &  &                                                                                                        S_w & S_v & &\dotsb\\
\dotsb & F_w \ar[ul] \ar[ur]^{\pr_w}\ar[urr]_-(0.5){\pr_v} &&& F_v \ar[ull]^-(0.5){\pr_w} \ar[ul]_{\pr_v} \ar[ur] &\dotsb
}
\end{equation}

% ------------------------------------------------
%	PROP

\begin{proposition}
\label{prop:net_state_sheaf}
The global sections of $\shf{N}$ are precisely the labelings of each vertex with feasible values for its respective control and state variables.
\end{proposition}

\begin{proof}
Given such a labeling, we have a value $s_v \in S_v$ for the state variables at each vertex and a value $c_v \in C_v$ for each control variable.
The $s_v$ specify values in each of the vertices of the base space for $\shf{N}$,
and there are no consistency checks needed downstream since this is the top level of the partial order.
However, given a $1$-hop neighborhood of vertex $v$,
there is precisely one way to assemble a value
$(c_v, s_v, s_{w_1}, s_{w_2}, \dotsc) \in C_v \times \prod_{w \in U_v} S_w$.
Notice that projection of this tuple onto each $S_w$ recovers the value of the labeling at $w$.
By assumption, this tuple is feasible,
and so is also an element of $F_v$.
Therefore, the labeling corresponds to a global section.

Conversely, given a global section of $\shf{N}$,
the values of the state variables are the values of the section at each vertex,
and the control variable values are the first component of the value of the section at each $1$-hop neighborhood.
Feasibility of the labeling follows from the consistency of the section and the choice of stalks at the $1$-hop neighborhoods in $\shf{N}$.
\end{proof}

%	PROP
% ------------------------------------------------

% --------------------------------------------------------------------------------------------------

\subsection{Performance against the objective}
\label{sec-sheaf-encoding_subsec-performance-obj}

To encode the problem of optimal control as the problem of finding a particular assignment to a sheaf,
we need to bring the objective functions $J_v$ into the specification of a sheaf.
We use $\shf{N}$ as a starting point and extend its partial order to include two new kinds of restriction maps
that capture both the individual $J_v$ and
the requirement that these functions be minimized.

Consider the same set $P = V \cup \left\{U_v : v \in V \right\}$ of vertices and $1$-hop neighborhoods of vertices,
but with the trivial partial order $\le'$ for which $x \le' y$ if and only if $x = y$ in $(P, \le')$.
Any sheaf on $(P, \le')$ has no nontrivial restrictions, since different elements are incomparable.
Let $\widehat{\mathbb{R}}$ be the sheaf on $(P,\le')$ in which every stalk is $\mathbb{R}$.
(Each restriction of $\widehat{\mathbb{R}}$ is the identity map $\mathbb{R} \to \mathbb{R}$ since it is along $x \le' x$.)
Every assignment of $\widehat{\mathbb{R}}$ is therefore a global section,
and consists of a choice of a real number label on each vertex and on each neighborhood.

The sheaf $\widehat{\mathbb{R}}$ has a zero section
(simply choose $0$ from every stalk)
and this section may also be thought of as the trivial sheaf $\widehat{0}$ defined in Section \ref{sec-basic_sheaves}.
There is a natural sheaf morphism $\widehat{0} \to \widehat{\mathbb{R}}$
in which every component map is the zero map.

The identity function $i \colon P \to P$ is trivially an order preserving function $\mbox{$(P,\le')$} \to \mbox{$(P,\le)$}$,
since the domain has no nontrivial chains.
We can use $i$ to collect the $J_v$ and $J'_v$ into a sheaf morphism $J \colon \shf{N} \to \widehat{\mathbb{R}}$,
with the $J_v$ being the component maps on vertices and the $J'_v$ being the component maps on the $1$-hop neighborhoods.

Extending the diagram for $\shf{N}$ \eqref{eqn:shf-N} yields the following,
denoted $\shf{M}$:
\begin{equation} \label{eqn:shf-M}
\xymatrix @R = 1pc @C = 2.5pc {
\widehat{0} \ar[d] & \dotsb & 0 \ar[d] & 0 \ar[d] & 0 \ar[d]  & 0 \ar[d] &\dotsb \\
\widehat{\mathbb{R}} & \dotsb & \mathbb{R}  &                                                                                                        \mathbb{R} &\mathbb{R}  & \mathbb{R} &\dotsb \\
\shf{N} \ar[u]^{J} & \dotsb &  &                                                                                                        S_w\ar[u]^{J_w} & S_v\ar[u]^{J_v} & &\dotsb \\
	& \dotsb & F_w\ar[ul] \ar[ur]^{\pr_w} \ar[uu]^{J'_w} \ar[urr]_-(0.5){\pr_v} &&& F_v \ar[uu]^{J'_v} \ar[ull]^-(0.5){\pr_w} \ar[ul]_{\pr_v} \ar[ur] &\dotsb
}
\end{equation}

% ------------------------------------------------
%	RMK

\begin{remark}
The diagram $\shf{M}$ \eqref{eqn:shf-M} is (trivially) a sheaf on a partial order.
Each stalk of $\widehat{\mathbb{R}}$ is above or incommensurate with every stalk of $\shf{N}$,
and the same holds for $\widehat{\mathbb{R}}$ and $\widehat{0}$.
As a result, there are no new commutativity constraints to check.
\end{remark}

%	RMK
% ------------------------------------------------

% ------------------------------------------------
%	PROP

\begin{proposition}
\label{prop:static_optimal}
The problem of finding an optimal control $c \in C_v$ for each vertex $v \in V$
is equivalent to the problem of finding an assignment $a$ that minimizes consistency radius on the sheaf $\shf{M}$
subject to the constraint that the restriction of $a$ to the sheaf $\shf{N}$ is a global section.
\end{proposition}

\begin{proof}
According to Proposition \ref{prop:net_state_sheaf},
the global sections on $\shf{N}$ correspond to feasible network states.
Specifically, if $a$ is an assignment that restricts to a global section on $\shf{N}$,
then it is a feasible network state.
Therefore, the consistency radius of $a$
is determined by the set of values
$\{(c_v,s_v, s_{w_1}, \dotsc) \in F_v : v \in V \}$.

A global section of $\shf{M}$ will result in zeros within $\widehat{\mathbb{R}}$ due to the zero morphism from $\widehat{0}$.
Since the sheaf morphism $J$ is essentially completely arbitrary,
the image of any given global section of $\shf{N}$ through the component maps of $J$ typically will not result in zeros on $\widehat{\mathbb{R}}$ -- and therefore not a global section on $\shf{M}$.
However, since $\widehat{\mathbb{R}}$ is a sheaf of pseudometric spaces,
we can compute consistency radius.
Since we are assuming a global section on $\shf{N}$,
the consistency radius is determined entirely by the assignment's values on $\widehat{\mathbb{R}}$.

The consistency radius for the assignment $a$ is the square root of the
sum of squares of differences of the form $|J'_v(c_v,s_v, s_{w_1}, \dotsc) - 0|$ and $|J_v(s_v) - 0|$,
where the $0$ terms are the images of the only assignment possible on $\widehat{0}$.
One merely needs to notice that the global objective function given by
\[ j((c_{v_1}, s_{v_1}), (c_{v_2}, s_{v_2}), \dotsc) := \sqrt{\sum_{v} \left( \left| J_v(s_v) \right|^2 + \left| J'_v(c_v, s_v, s_{w_1}, \dotsc ) \right|^2 \right)} \]
is the consistency radius,
so minimizing the consistency radius is precisely the problem of optimal control.
\end{proof}

%	PROP
% ------------------------------------------------

As a consequence, the consistency radius should be interpreted as an aggregation of the residuals of the optimal control problem.
The local consistency radii capture the residuals on various subsets of vertices.

It may be dramatically easier to minimize the consistency radius over all of $\shf{M}$ without worrying about whether we have a global section on $\shf{N}$ (i.e., a fully consistent network state).
The local consistency radius on $\shf{N}$ then tells us by how much this ``relaxed'' solution differs from the fully-consistent one.

% ------------------------------------------------
%	PROP

\begin{proposition}
\label{prop:static_sheafrelaxed_optimal}
The assignment $b$ that minimizes the consistency radius of the sheaf $\shf{M}$
without the constraint that $b$ restrict to a global section of the sheaf $\shf{N}$
differs from the true optimal control by at least a constant factor times its local consistency radius on $P$
(the base space of $\shf{N}$).
\end{proposition}

\begin{proof}
Let $b$ be an assignment that minimizes consistency radius on $\shf{M}$ without further constraints
and suppose that $a$ is an assignment on $\shf{M}$ that represents the solution to the optimal control problem.
Since $\shf{N}$ is defined on a subspace of the base space of $\shf{M}$,
$b$ restricts to an assignment $b'$ that minimizes consistency radius on $\shf{N}$ as well.
By Proposition \ref{prop:static_optimal},
$a$ restricts to a global section $a'$ on $\shf{N}$.
Since each restriction map in $\shf{N}$ is a projection and thus has Lipschitz constant $1$,
we can apply Proposition \ref{sec-sheaf-encoding_prop-crbound} with $K=1$ to see that
\[ c_{\shf{N}}(b') \le (1+1) d_{\shf{N}}(a',b') = 2 d_{\shf{N}}(a',b') \le 2 d_{\shf{M}}(a,b). \]
The final equality above arises because the assignment pseudometric on $\shf{M}$
contains a sum over a strictly larger set than that of $\shf{N}$.
Since $a$ is the solution to the optimal control problem,
$d_\shf{M}(a,b)$ is the distance between the solution to the optimal control problem
and its estimate $b$ found by minimizing consistency radius.
\end{proof}

%	PROP
% ------------------------------------------------

The proof of Proposition \ref{prop:static_sheafrelaxed_optimal} ignores the structure of the $J$ morphism,
which implies that the bound obtained in the proof is loose.
The residual for the optimal control problem, $c_{\shf{M}}(a)$,
may be rather different from the residual for the sheaf-relaxed version, $c_{\shf{M}}(b)$.
Monotonicity for local consistency radii gives that
$0 = c_{\shf{N}}(a) \le c_{\shf{M}}(a)$
and
$c_{\shf{N}}(b) \le c_{\shf{M}}(b)$,
but since $a$ is not a global section on $\shf{M}$,
we cannot assert that $c_{\shf{M}}(b)$ is related to it.

% --------------------------------------------------------------------------------------------------

\subsection{Evolving the dynamics}
\label{sec-sheaf-encoding_subsec-evolving-dynamics}

Consider the subset $U \subset P$ consisting of only the $1$-hop neighborhoods of each vertex.
Define a sheaf $\shf{L}$ on $U$ stalkwise with $\shf{L}(U_v) := S_v$ for each vertex $v$,
noting that there are no restrictions to be defined.

Define a sheaf morphism $p \colon \shf{N} \to \shf{L}$ along the inclusion of $U$ into $P$ whose component maps are given by the projections $p_v := \pr_{S_v}$.
In the absence of control variables (i.e., $C_v=\emptyset$ for all $v$),
then $p$ induces a bijection on the space of global sections.

The sheaf morphism for the dynamics $f \colon \shf{N} \to \shf{L}$ is constructed similarly to $p$.
In this case,
the component map at a $1$-hop neighborhood $U_v$ is given by $f_v$.

% ------------------------------------------------
%	PROP

\begin{proposition}
\label{prop:dynamic_sheaf}
If we construct the diagram of sheaf morphisms from copies $\shf{L}_n := \shf{L}$ and $\shf{N}_{n} := \shf{N}$,
\[
\xymatrix{
\dotsb \ar[r] & \shf{L}_n & \shf{N}_n \ar[l]_p \ar[r]^f & \shf{L}_{n+1} & \shf{N}_{n+1} \ar[l]_p \ar[r]^f & \dotsb
}
\]
and reinterpret this as a new sheaf $\shf{T}$,
then global sections of $\shf{T}$ are precisely feasible trajectories of the network's state given control actions at each time step.
\end{proposition}

\begin{proof}
We merely need to leverage Proposition \ref{prop:net_state_sheaf}
and employ the dynamics at each time step by interpreting $f_v$ as component maps of a sheaf morphism.
\end{proof}

%	PROP
% ------------------------------------------------

The diagram for a single time step within the sheaf $\shf{T}$ is somewhat large,
but is ultimately straightforward to decompose:
\[
\xymatrix @R = 0.6pc @C = 2.5pc {
\shf{N}_{n} \ar[dd]_f & \dotsb & &                                                                                                         S_w & S_v && \dotsb\\
& \dotsb & F_w \ar[dr]^{f_w} \ar[ul] \ar[ur]^{\pr_w} \ar[urr]_-(0.5){\pr_v} &&& F_v \ar[dl]_{f_v} \ar[ull]^-(0.5){\pr_w} \ar[ul]_{\pr_v} \ar[ur] &\dotsb\\
\shf{L}_{n+1} & \dotsb   && S_w  & S_v &  &\dotsb \\
& \dotsb & F_w \ar[ur]^{p_w} \ar[dl] \ar[dr]_{\pr_w} \ar[drr]^-(0.5){\pr_v} &&& F_v \ar[ul]_{p_v} \ar[dll]_-(0.5){\pr_w} \ar[dl]^{\pr_v} \ar[dr] &\dotsb\\
\shf{N}_{n+1}\ar[uu]^p& \dotsb &    &                                                                                                      S_w & S_v && \dotsb
}
\]

% --------------------------------------------------------------------------------------------------

\subsection{Optimal control as a diagram of sheaves}
\label{sec-sheaf-encoding_subsec-opt-control-sheaves}

Given our constructions thus far,
we may simply compose diagrams to obtain the optimal control problem through
finitely many time steps as a consistency radius minimization.
We emphasize that the correct interpretation is that the objective functions are taken to be nonnegative,
with $0$ being the goal.

% ------------------------------------------------
%	THM

\begin{theorem}
\label{thm:dynamic_optimal}
Consider the problem of optimal control of the network for finitely many time steps.
If we aggregate the sheaves from the previous sections into one big diagram,
itself a sheaf $\shf{S}$,
constructed as the diagram
\begin{equation}
\label{eqn:dynamic_optimal}
\xymatrix @R = 1pc {
\dotsb \ar[r] & \shf{L}_n & \shf{N}_n \ar[l]_p \ar[r]^f \ar[d]^J & \shf{L}_{n+1} & \shf{N}_{n+1} \ar[l]_p \ar[r]^f \ar[d]^J & \dotsb \\
                   &                & \widehat{\mathbb{R}}                &                      & \widehat{\mathbb{R}}                        &           \\
                   &                & \widehat{0} \ar[u]                &                      & \widehat{0} \ar[u]                      &
}
\end{equation}
and minimize consistency radius over this diagram under the constraint that we manipulate global sections of $\shf{N}_\bullet$,
then this is precisely the same as solving an optimal control problem.
\end{theorem}

\begin{proof}
The proof is just a restatement of Propositions \ref{prop:static_optimal} and \ref{prop:dynamic_sheaf} on a larger diagram.
The sums in defining the objective functions%
\footnote{If we wish to consider infinitely many time steps, then we can replace the sums with an infinite series or a supremum.}
remain finite under this situation since there are only finitely-many time steps involved.
\end{proof}

%	THM
% ------------------------------------------------

Again, we get the added benefit of a relaxed version of the problem.
If we minimize consistency radius without constraints instead of solving the problem of minimizing consistency radius subject to a portion being a global section,
then this is likely much easier.
However, we can track the error in this relaxation simply by computing the local consistency radius on the top row of the diagram for $\shf{S}$ \eqref{eqn:dynamic_optimal}.
The constant relating the consistency radius to the accumulated error consists of the aggregate of the Lipschitz constants of all of the $f_v$, $J_v$, and $J'_v$ functions.

%% file: btm-relaxation.tex
% !TeX root=./control-master.tex

\section{Relaxing to a Boolean model}
\label{sec-btm-relaxation}

In its general form as described in Section \ref{sec-problem-statement},
the problem of optimal control can be computationally intensive, if not impossible, to solve.
Typical approaches to solving such problems rely on carefully-chosen relaxation schemes
that allow efficient solutions to difficult problems,
generally at the expense of accuracy.
In this section,
we turn our attention to a class of optimal control problems that naturally facilitates a Boolean relaxation scheme,
and analyze these problems with the methods developed in Section \ref{sec-sheaf-encoding}.
Such problems are attractive because their relaxed Boolean forms can often be expressed as integer programs,
for which efficient solution algorithms exist.
An optimal solution to the Boolean-relaxed problem can be translated into a solution to the original problem,
and while such a solution is not necessarily optimal in the original domain,
it is often ``good enough.''

We note that the methods utilized in this section focus on Boolean relaxations,
but they are quite general and can be applied to other discrete relaxations.

% --------------------------------------------------------------------------------------------------

\subsection{Boolean discretization}
\label{sec-btm-relaxation_subsec-thresholding}

Recall that a space of endogenous state variables $S_v$ and a space of exogenous control variables $C_v$ are associated with each piece of equipment $v$.
We assume that some subset of the state space $S_v$ is considered to contain ``operational'' states,
and that the values in the complement of that subset are considered ``failed.''
Similarly,
each value in the control space $C_v$ can be discretized to one of two options:
``apply control to $v$'' or ``do not apply control to $v$.''

Formally, for each piece of equipment $v \in V$,
let $\wtSv := \{0,1\}^{\dim(S_v)}$ be the Boolean version of the state space $S_v$
and let $\wtCv := \{0,1\}^{\dim(C_v)}$ be the Boolean version of the control space $C_v$.
Intuitively, a value of $1$ means that a piece of equipment is operational,
while a value of $0$ means that it is failed.
We can combine these spaces to create the Boolean version of the system state space $R_v$,
$\wtRv := \wtCv \times \prod_{w \in U_v} \widetilde{S_w}$.
As in the original case,
it is possible that not every Boolean system state in $\wtRv$ has a sensible interpretation,
so we use $\wtFv \subseteq \wtRv$ to denote the space of feasible Boolean system states for the vertex $v$.
While it is not necessary that $\wtFv$ be functionally related to $F_v$,
we require that functions defined between the original and Boolean spaces \emph{preserve feasibility},
i.e., map feasible inputs to feasible outputs.
Feasibility preservation is a reasonable requirement that is likely to be satisfied by most sensible discretization schemes,
since a feasible solution to the Boolean-relaxed problem is not very useful if it does not correspond to a feasible analogue in the original problem domain,
and it would be unintuitive for a feasible solution in the original domain to have an infeasible Boolean relaxation.

In order to compare system states and their relaxed counterparts,
we require measurable,%
\footnote{Recall from Section \ref{sec-problem-statement_subsec-dynamical-system} that $C_v$ and $S_v$ are pseudometric spaces for all $v \in V$.}
feasibility-preserving functions that map between the original spaces and their corresponding Boolean spaces.

% ------------------------------------------------
%	DEF

\begin{definition} \label{def:thresholding-functions}
For each piece of equipment $v \in V$,
select measurable functions
$\tau_v \colon S_v \to \wtSv$ and $\chi_v \colon C_v \to \wtCv$
such that
$\tau_v(\pr_{S_v}(F_v)) \subseteq \pr_{\wtSv}(\wtFv)$ and that
$\chi_v(\pr_{C_v}(F_v)) \subseteq \pr_{\wtCv}(\wtFv)$.
Extend $\chi_v$ and $\{ \tau_w : w \in U_v \}$ to the function $\sigma_v \colon R_v \to \wtRv$ defined as
\begin{equation} \label{eqn:sigma_v}
\sigma_v(c_v, s_v, s_{w_1}, \dotsc) := (\chi_v(c_v), \tau_v(s_v), \tau_{w_1}(s_{w_1}), \dotsc).
\end{equation}
The functions $\tau_v$, $\chi_v$, and $\sigma_v$ are called \emph{(Boolean) thresholding functions}.
\end{definition}

% ------------------------------------------------

\begin{definition} \label{def:lifting-function}
For a piece of equipment $v \in V$,
a \emph{lifting function} is any measurable function
$\gamma_v \colon \wtRv \to R_v$
such that
$\gamma_v( \wtFv ) \subseteq F_v$.
\end{definition}

%	DEF
% ------------------------------------------------

We are free to define the functions $\tau_v$, $\chi_v$, and $\gamma_v$ in any sensible way,
although the choice of thresholding functions will often be dictated by the structure of the original problem and the desired Boolean relaxation.
Note that measurability and feasibility preservation for $\tau_v$ and $\chi_v$ imply that $\sigma_v$ is also measurable and feasibility-preserving.
In most cases,
a given $\sigma_v$ will be a many-to-one function,
but $\gamma_v$ should be injective.
Thresholding usually implies a loss of fidelity,
since although we might be able to arrange for
$\sigma_v \circ \gamma_v|_{\wtFv} = \id_{\wtFv}$,
rarely will we have $\gamma_v \circ \sigma_v|_{F_v} = \id_{F_v}$.

The final component of the relaxation scheme is a set of dynamics functions that govern how the system behaves in the Boolean space.
For each piece of equipment $v \in V$,
we choose a measurable \emph{Boolean dynamics function} $\wtfv \colon \wtRv \to \wtSv$
that serves as an approximation of the original dynamics $f_v$ in the Boolean space.
(Recall that since each $f_v$ is assumed to be a Lipschitz function,
it is automatically Borel measurable.)
Analogously to the original dynamics function $f_v$,
we require that feasible Boolean system states evolve to feasible states:
$\wtfv( \wtFv ) \subseteq \pr_{\wtSv}( \wtFv )$.
Selection of an appropriate $\wtfv$ is critically important to the success of the relaxation,
and can unfortunately be difficult in the case of complex systems with poorly-understood dynamics.

Since our goal is to use the Boolean-relaxed system to efficiently identify candidate solutions to the original problem,
quantifying the error incurred by the thresholding process is of interest.
If this error is large,
then any performance benefit obtained by relaxation will be wasted by solving the wrong problem!
To assist with the error analysis,
consider the following diagram,
which demonstrates the interaction of the original and relaxed spaces and functions:
\begin{equation} \label{eqn:btm-stalk-diagram}
\xymatrix @R = 1.5pc @C = 3pc {
F_v \ar[d]_-{f_v} \ar@/_/[r]_{\sigma_v} & \wtFv \ar@/_/[l]_-{\gamma_v} \ar[d]^-{\wtfv}\\
S_v \ar[r]_{\tau_v}  & \wtSv\\
}
\end{equation}

In general, the diagram \eqref{eqn:btm-stalk-diagram} will not commute,
and a lack of commutativity along the thresholding functions will induce error in our relaxation scheme.
Specifically, if thresholding could be performed without error,
then the values $(\tau_v \circ f_v)(x)$ and $(\wtfv \circ \sigma_v)(x)$ would be equal for all $x \in F_v$;
the interpretation in this case is that it does not matter whether we
propagate $x$ in the original space and threshold our result,
or threshold $x$ and then propagate in the Boolean space.

To quantify error requires a metric,
and since $\wtFv$ and $\wtSv$ are vector spaces,
selecting a metric induced by a norm on each is appropriate.
Furthermore, since each function under discussion is measurable,
we may select a well-defined operator norm for each space of functions;
this includes the spaces of functions whose domains are $F_v$ or $S_v$,
such as $\sigma_v$ and $\tau_v$.
With this in mind,
we turn our attention to bounding the worst-case thresholding error,
$\| \wtfv \circ \sigma_v|_{F_v} - \tau_v \circ f_v|_{F_v} \|$.

% ------------------------------------------------
%	PROP

\begin{proposition} \label{prop:thresholding-error-bound}
Fix $v \in V$ and let $\sigma_v$ and $\tau_v$ be Boolean thresholding functions, $\gamma_v$ be a lifting function, and $\wtfv$ be a Boolean dynamics function.
If $(\omega_1)_v := \|\wtfv - \tau_v \circ f_v \circ \gamma_v|_{\wtFv}\|$
and $(\omega_2)_v := \|\gamma_v \circ \sigma_v|_{F_v} - \id_{F_v}\|$,
then
\[
\| \wtfv \circ \sigma_v|_{F_v} - \tau_v \circ f_v|_{F_v} \| \leq (\omega_1)_v \|\sigma_v|_{F_v}\| + (\omega_2)_v \|\tau_v \circ f_v|_{F_v}\|.
\]
\end{proposition}

\begin{proof}
In the interest of saving space,
we abuse notation below
by leaving off the restrictions of the domains of each function to $F_v$.
We see that
\begin{align*}
\| \wtfv \circ \sigma_v - \tau_v \circ f_v \|
	&= \| \wtfv \circ \sigma_v - \tau_v \circ f_v \circ \gamma_v \circ \sigma_v + \tau_v \circ f_v \circ \gamma_v \circ \sigma_v - \tau_v \circ f_v \| \\
%	&\leq \| \wtfv \circ \sigma_v - \tau_v \circ f_v \circ \gamma_v \circ \sigma_v \| + \| \tau_v \circ f_v \circ \gamma_v \circ \sigma_v - \tau_v \circ f_v \| \\
	&\leq \|\wtfv - \tau_v \circ f_v \circ \gamma_v\| \|\sigma_v\| + \|\tau_v \circ f_v\| \|\gamma_v \circ \sigma_v - \id_{F_v}\|. \qedhere
\end{align*}
\end{proof}

%	PROP
% ------------------------------------------------

In general,
the definitions of the thresholding functions will be dictated by the underlying problem and the desired relaxation,
but a modeler has some flexibility in choosing $\gamma_v$ and $\wtfv$.
An immediate corollary to Proposition \ref{prop:thresholding-error-bound}
suggests optimal choices for these functions.

% ------------------------------------------------
%	COR

\begin{corollary} \label{cor:optimal-wtfv}
Let $v \in V$ be a piece of equipment and suppose that the thresholding functions $\tau_v$ and $\sigma_v$ are fixed.
If $\wtfv$ and $\gamma_v$ are chosen such that
$\wtfv = \tau_v \circ f_v \circ \sigma_v$ and $\gamma_v \circ \sigma_v|_{F_v} = \id_{F_v}$,
then $\|\wtfv \circ \sigma_v|_{F_v}- \tau_v \circ f_v|_{F_v} \| = 0$,
i.e., no error is incurred by the thresholding process.
\end{corollary}

%	COR
% ------------------------------------------------

To summarize, Proposition \ref{prop:thresholding-error-bound} relates the error incurred by thresholding
to the error in approximating the dynamics ($(\omega_1)_v$)
and the error in recovering an original state from its discretized form
($(\omega_2)_v$).
Each of these quantities can be minimized by selecting appropriate $\wtfv$ and $\gamma_v$ functions,
and Corollary \ref{cor:optimal-wtfv} provides sufficient conditions for ensuring zero thresholding error.
Note that the choices for $\wtfv$ and $\gamma_v$ suggested by the corollary are optimistic.
Choosing $\wtfv = \tau_v \circ f_v \circ \sigma_v$ may not be possible,
as $f_v$ may not be fully known to the modeler,
or even preferable,
as a fully-known $f_v$ may be computationally prohibitive to evaluate.
Similarly,
$\sigma_v$ will typically be a many-to-one function on $R_v$,
so the definition of $F_v$ and the restriction of $\sigma_v$ to that subspace
thus determine whether or not $\gamma_v$ can be chosen per the corollary.

% --------------------------------------------------------------------------------------------------

\subsection{Lifting the thresholding into the sheaf}
\label{sec-btm-relaxation_subsec-lifting-thresh}

By thresholding each stalk of the sheaf $\shf{N}$ using the process described in the previous section,
this sheaf is transformed into a new sheaf $\wtshfN$ whose stalks are spaces of Boolean vectors.
Using the construction given in \cite[Section 4]{robinson2017sheaf},
the thresholding process can be realized as three sheaf morphisms,
$\Sigma \colon \shf{N} \to \wtshfN$,
$\Gamma \colon \wtshfN \to \shf{N}$, and
$T \colon \shf{L} \to \wtshfL$,
that use the $\sigma_v$, $\gamma_v$, and $\tau_v$ functions defined in the previous section as the component maps for each stalk.
For each $1$-hop neighborhood $U_v$,
let $(\Sigma)_{U_v} := \sigma_v$, $(\Gamma)_{U_v} := \gamma_v$, and $(T)_{U_v} := \tau_v$,
and for a vertex $v$,
let $(\Sigma)_{v} := \tau_v$ and $(\Gamma)_{v} := \gamma_v |_{\wtSv}$.
The latter is well-defined because the domain of $\gamma_v$ is $\wtRv$,
a product for which one factor is $\wtSv$.
Since $\shf{L}$ only has stalks on the $1$-hop neighborhoods,
$T$ does not have components on $V$.

An analogous sheaf morphism for the Boolean dynamics,
$\widetilde{f} \colon \wtshfN \to \wtshfL$,
can be defined by using the Boolean dynamics functions $\wtfv$ described in Section \ref{sec-btm-relaxation_subsec-thresholding} stalk-wise:
$\widetilde{f_{U_v}} := \wtfv$.
Although this notation appears to be slightly inconsistent,
it is not ambiguous because the sheaf morphism $\widetilde{f}$ only has component maps on the $1$-hop neighborhoods $U_v$.

In the neighborhood of $U_v$, the diagram of these sheaf morphisms has the form
\begin{equation} \label{eqn:btm-sheaf}
\xymatrix @R = 1pc {
\shf{N}\ar@/_1pc/[r]_{\Sigma} \ar[dd]_{f}& \wtshfN\ar@/_1pc/[l]_{\Gamma} \ar[dd]_{\widetilde{f}} & S_w\ar@/^2pc/[rrr]^{\tau_w}&& S_v\ar@/^2pc/[rrr]^{\tau_v}                                                                                                  &       \widetilde{S_w} \ar@/_4pc/[lll]_{\gamma_w|_{\wtSv}} && \wtSv \ar@/_4pc/[lll]_{\gamma_v|_{\wtSv}} &\\
& && F_v \ar@/_1pc/[rrr]_{\sigma_v}\ar[d]^{f_{U_v}} \ar[ul]^{\pr_w} \ar[ur]^{\pr_v}  &&&\wtFv \ar@/_1pc/[lll]_{\gamma_v} \ar[d]^{\widetilde{f_{U_v}}} \ar[ur] \ar[ul] &\\
\shf{L}\ar[r]_{T}& \wtshfL  & & S_v \ar[rrr]_{\tau_v} &&& \wtSv  &\\
}
\end{equation}
This diagram of sheaf morphisms is approximately commutative,
such that the difference between $T \circ f$ and $\widetilde{f} \circ \Sigma$ is a bounded function from the global sections of
$\shf{N}$ to the global sections of $\wtshfN$ in the assignment pseudometric.

Since Section \ref{sec-sheaf-encoding} established that the solutions to the optimal control problem correspond to certain sheaf assignments with minimal consistency radius,
we need to determine the impact of the approximation error on the consistency radius of a typical assignment.
With this in hand, we can determine the impact of this approximation error on the solutions to the optimal control problem.

Recall that the sheaf $\shf{S}$ \eqref{eqn:dynamic_optimal} defined in Theorem \ref{thm:dynamic_optimal}
encodes solutions to the optimal control problem as assignments
that restrict to sections on portions of the diagram.
We may repeat the construction shown in \eqref{eqn:btm-sheaf}
in each time step to build a Boolean thresholded sheaf $\wtshfS$,
whose diagram matches \eqref{eqn:dynamic_optimal}
with $\shf{N}_\bullet$, $\shf{L}_\bullet$, $p$, $f$, and $J$ replaced with their Boolean counterparts;
in this case,
the components of the $\widetilde{p}$ and $\widetilde{J}$ morphisms are given by composing
the components of $p$ and $J$ with $\sigma_v$ or $\tau_v$ as appropriate.

If $r$ is an assignment of $\wtshfS$,
then it restricts to an assignment on each $\wtshfNn$ as well.
Applying the sheaf morphism $\Gamma$ on each such subsheaf,
we can translate $r$ into an assignment $\Gamma(r)$ on each $\shf{N}_n$.

It is perhaps useful to mention that the failure of the diagrams of thresholding morphisms to commute
will typically destroy global sections in the original sheaves when they are thresholded.
Therefore, the sheaf-based relaxations are essential to obtain bounds.
Proposition \ref{sec-sheaf-encoding_prop-morphism} states that the consistency radius of the image of an assignment through a sheaf morphism
is bounded by the consistency radius of the assignment in the domain.
The bound on the approximation error from Proposition \ref{prop:thresholding-error-bound} can be aggregated across all vertices,
and adds an extra term to this bound on consistency radius.

% ------------------------------------------------
%	THM

\begin{theorem} \label{sec-btm-relaxation_subsec-lifting-thresh-thm}
Let $r$ be an assignment to the sheaf $\wtshfS$
and let $s$ be an assignment of the sheaf $\shf{S}$
that restricts to a section on each $\shf{N}_n$.
Suppose that we take
\[
\epsilon_v := (\omega_1)_v \|\sigma_v\|  + (\omega_2)_v \|\tau_v \circ f_v \|,
\]
where $(\omega_1)_v := \|\wtfv - \tau_v \circ f_v \circ \gamma_v \|$ and $(\omega_2)_v := \|\gamma_v \circ \sigma_v - \id_{F_v}\|$ as in Proposition \ref{prop:thresholding-error-bound},
restricting domains to $F_v$ as appropriate.
Then
\[
c_{\wtshfNn}(r) \le K c_{\shf{N}_n}(\Gamma(r)) + C \epsilon \le 2 K d_{\shf{N}_n}(s,\Gamma(r)) + C \epsilon,
\]
where $\epsilon := \max_{v \in V}\{ \epsilon_v \}$,
$C^2$ is the number of restrictions in $\shf{N}_n$,
and $K$ is the largest Lipschitz constant of all of the $\gamma_v$, $f_v$, $J_v$, and $J'_v$ maps.
\end{theorem}

If $r$ is the result of minimizing the consistency radius of an assignment of $\wtshfS$,
then $r$ can be interpreted as solving the Boolean thresholded problem rather than the original optimal control problem.
On the other hand, if the assignment $s$ is constructed according to the recipe in Proposition \ref{prop:static_optimal} and Theorem \ref{thm:dynamic_optimal},
then $s$ is an encoding of the solution to the original optimal control problem.
Theorem \ref{sec-btm-relaxation_subsec-lifting-thresh-thm} claims that the consistency radius of $r$ is a bound on the difference between
the solution to Boolean thresholded problem and the original optimal control problem.

\begin{proof}
First of all, since $r$ is an assignment of $\wtshfS$,
we can translate $r$ into an assignment $\Gamma(r)$ on each $\shf{N}_n$
as described before the statement of the Theorem.

On the other hand, we have assumed that $s$ is a section when it is restricted to each of the $\shf{N}_n$,
so we can compare $\Gamma(r)$ and $s$ via Proposition \ref{sec-sheaf-encoding_prop-crbound} to obtain
\begin{equation}
\label{eqn:thm2_eq1}
c_{\shf{N}_n}(\Gamma(r)) \leq 2 d_{\shf{N}_n}(s,\Gamma(r)),
\end{equation}
noting that all of the restrictions in $\shf{N}_n$ are projections and therefore have Lipschitz constant $1$.

By Proposition \ref{prop:thresholding-error-bound},
the maximum difference between the upper ($\tau_v \circ f_v$) and lower ($\wtfv \circ \sigma_v$) paths in the diagram \eqref{eqn:btm-stalk-diagram}
is bounded by $\epsilon_v$ for each vertex $v \in V$.
Therefore,
$\epsilon = \max_{v \in V} \{\epsilon_v\}$ bounds the failure of the thresholding maps to be a sheaf morphism across all stalks.
We can invoke Proposition \ref{sec-sheaf-encoding_prop-morphism} to conclude that
\begin{equation}
\label{eqn:thm2_eq2}
c_{\wtshfNn}(r) \le K c_{\shf{N}_n}(\Gamma(r)) + C \epsilon.
\end{equation}
Notice in particular that the components of $\Gamma$ are included in the bound $K$ on the Lipschitz constants.

Combining the two inequalities \eqref{eqn:thm2_eq1} and \eqref{eqn:thm2_eq2} yields
\[
c_{\wtshfNn}(r) \le K c_{\shf{N}_n}(\Gamma(r)) + C \epsilon \le 2 K d_{\shf{N}_n}(s,\Gamma(r)) + C \epsilon. \qedhere
\]
\end{proof}

%	THM
% ------------------------------------------------

This bound is rather pessimistic,
in that as the discretization error $d_{\shf{N}_n}(s,\Gamma(r))$ incurred by converting the problem to a Boolean thresholded one increases,
the consistency radius $c_{\wtshfNn}(r)$ computed on the thresholded sheaf tells less about the original problem.
On the other hand,
it is still the case that according to Theorem \ref{thm:dynamic_optimal},
we can assess how well the thresholded model is doing against the original problem.

% ------------------------------------------------
%	COR

\begin{corollary} \label{cor:gen_static_sheafrelaxed_optimal}
If the thresholding error is zero (i.e., $\epsilon = 0$),
then Theorem \ref{sec-btm-relaxation_subsec-lifting-thresh-thm} reduces to a time-dependent generalization of Proposition \ref{prop:static_sheafrelaxed_optimal}.
\end{corollary}

%	COR
% ------------------------------------------------

%% file: application-boolean-control.tex
% !TeX root=./control-master.tex

% --------------------------------------------------------------------------------------------------

\section{Application: Boolean state control}
\label{sec-appl-boolean-control}

Per the discussion in Section \ref{sec-btm-relaxation_subsec-thresholding},
suppose that each piece of equipment (node) is endowed with both ``operational'' and ``failed'' states,
and that each node has a switching mechanism (control) that,
when activated,
may toggle the node's state between these two modes.
We do not assume that activating the control will necessarily induce a change in state,
as network topology and other physical considerations may preclude a state change from occurring.
As an example,
the on/off state of an incandescent light fixture in a home can typically be controlled by a corresponding light switch on the same circuit.
If the breaker for that circuit is tripped,
then the light fixture will become stuck in the off state,
regardless of the position of the light switch.
We note that Boolean state control has many useful applications.
Such models are of practical interest because of their ability to describe problems across domains as diverse as
biological networks \cite{claleealobuspoo18},
efficient marketing of products in social networks \cite{easkle10},
and transportation network fragility \cite{keafar18,com17}.
Also of interest in applications are network control systems that exhibit linear quadratic regulator (LQR) dynamics,
which we discuss in more detail in Section \ref{sec-appl-boolean-control_subsec-affine}.

% --------------------------------------------------------------------------------------------------

\subsection{Problem setup}

Boolean state control problems are ones in which state changes are induced by
applying a control action at a subset of nodes in the network.
The immediate state changes caused by the controls may subsequently cascade through the network according to its specific dynamics.
We associate a Boolean control space $C_v := \{c_0, c_1\}$ to each node $v$ of the network,
where $c_0$ denotes that no control action will be applied to the node,
and $c_1$ denotes that a control action will be applied.
In the case of our lighting example,
the control action for the light fixture is the act of flipping the light switch on,
and the control space represents the choice of whether or not to carry out that action.

At each node $v$,
we partition the state space $S_v$ into two disjoint sets, $\Omega_v$ and $\Phi_v$,
representing the operational and failed states, respectively, for $v$.
In practice,
a description of the operational set $\Omega_v$ is typically given,
in which case the failed set $\Phi_v$ is defined as the set difference $\Phi_v := S_v - \Omega_v$.
Partitioning the state space allows us to unambiguously classify each possible node state as operational or failed,
which fits into our desired Boolean model.

In order to allow us to assess the feasibility of a given system state,
we further refine this model.
Assume that the state space $S_v$ for each piece of equipment $v$ contains a pair of \emph{nominal states},
denoted $(s_\Phi)_v$ and $(s_\Omega)_v$.
The value $(s_\Phi)_v \in \Phi_v$ represents the ideal state value that $v$ will take when it is failed,
and $(s_\Omega)_v \in \Omega_v$ represents the ideal state value that $v$ will take when it is operational.
In the context of the lighting example,
if the node state $s_v \in S_v$ is defined as the voltage at the light fixture,
then the nominal failed state $(s_\Phi)_v$ could correspond to electrical ground ($0$V),
while the nominal operational state $(s_\Omega)_v$ could correspond to the light fixture's normal operating voltage, say, $120$V.

% ------------------------------------------------
%	DEF

\begin{definition}\label{def:nominal_configuration}
Suppose that nominal states $\{(s_\Phi)_v, (s_\Omega)_v\} \subseteq S_v$ are defined at each node $v$ and let $s \in \prod_{v \in V} S_v$.
If $\pr_{S_v} s \in \{(s_\Phi)_v, (s_\Omega)_v\}$ for all $v$,
then $s$ is a \emph{nominal configuration}.
Let $\mathbb{S}$ denote the set of all nominal configurations.
\end{definition}

%	DEF
% ------------------------------------------------

We say that the network is in a feasible state if all of its node states collectively form a nominal configuration.
As a consequence,
we have $F_v = C_v \times \pr_{U_v} \mathbb{S}$.

The theoretical framework also requires us to identify $\wtFv$, the space of Boolean feasible system states at node $v$.
In words,
this will be defined as the set of all Boolean states that correspond to the discretization of a feasible system state;
mathematically, $\wtFv = \sigma_v(F_v)$,
where $\sigma_v$ is the thresholding function \eqref{eqn:sigma_v} for the Boolean state control problem,
explicitly defined in the next section.

% --------------------------------------------------------------------------------------------------

\subsection{Boolean discretization}
\label{sec-appl-boolean-control_subsec-boolean-discretization}

For the Boolean discretization,
we first define the thresholding functions $\chi_v$, $\tau_v$, and $\sigma_v$ per Definition \ref{def:thresholding-functions}.
Let $\wtCv := \{0,1\}$ and define
$\chi_v \colon C_v \to \wtCv$ element-wise by $\chi_v(c_0) := 0$ and $\chi_v(c_1) := 1$.
Similarly,
let $\wtSv := \{0,1\}$ and define
$\tau_v \colon S_v \to \wtSv$ to take the value $0$ on $\Phi_v$ and the value $1$ on $\Omega_v$.
Recall that the final thresholding function, $\sigma_v \colon R_v \to \wtRv$,
is explicitly constructed in terms of $\chi_v$ and $\tau_v$ in \eqref{eqn:sigma_v}.
Note that the definition $\wtFv = \sigma_v(F_v)$ implies that feasibility preservation is satisfied for these functions.

Next,
we turn our attention to the lifting function $\gamma_v$ given in Definition \ref{def:lifting-function}.
The control thresholding function $\chi_v$ is invertible by construction,
and we use $\chi_v^{-1} \colon \wtCv \to C_v$ to represent the inverse function.
In contrast,
the state thresholding function $\tau_v$ is generally not invertible.
In lieu of a true inverse,
we introduce a function $\rho_v \colon \wtSv \to S_v$ that maps a Boolean state value to the corresponding nominal state value;
explicitly,
$\rho_v(0) := (s_\Phi)_v$ and $\rho_v(1) := (s_\Omega)_v$.
With these in hand,
the lifting function $\gamma_v \colon \wtRv \to R_v$ can be defined as
\begin{equation} \label{eqn:appl-gamma_v}
\gamma_v(\wtcv, \wtsv, \widetilde{s_{w_1}}, \dotsc) := (\chi_v^{-1}(\wtcv), \rho_v(\wtsv),\rho_{w_1}(\widetilde{s_{w_1}}), \dotsc).
\end{equation}
Feasibility preservation for $\gamma_v$ is an immediate corollary of the next result.

% ------------------------------------------------
%	PROP

\begin{proposition} \label{prop:boolean_left_inverse}
Let $v \in V$.
The thresholding function $\sigma_v$ and lifting function $\gamma_v$
defined (in this section) for the Boolean state control problem satisfy
$\gamma_v \circ \sigma_v|_{F_v} = \id_{F_v}$.
\end{proposition}

\begin{proof}
Let $x := (c_v, s_v, s_{w_1}, \dotsc) \in F_v$,
so $\pr_{S_w} (x) \in \{ (s_\Phi)_w, (s_\Omega)_w \}$ for each $w \in U_v$.
Note also that
$\tau_v((s_b)_v) = b$ for $b \in \{0,1\}$,
per the definition of nominal state.
The composition $\gamma_v \circ \sigma_v$ satisfies
\[
(\gamma_v \circ \sigma_v)|_{F_v}(x) = ((\chi^{-1}_v \circ \chi_v)(c_v), (\rho_v \circ \tau_v)(s_v), (\rho_{w_1} \circ \tau_{w_1})(s_{w_1}), \dotsc ),
\]
and since $\chi^{-1}_v \circ \chi_v = \id_{C_v}$ and
$(\rho_w \circ \tau_w)((s_b)_w) = \rho_w(b) = (s_b)_w$ for $b \in \{0,1\}$,
the result is immediate.
\end{proof}

%	PROP
% ------------------------------------------------

While feasibility preservation for $\gamma_v$ is a consequence of Proposition \ref{prop:boolean_left_inverse},
more interesting is how this result relates to the analysis of thresholding error for the Boolean state control problem.
Referring to Proposition \ref{prop:thresholding-error-bound},
we conclude that the discretization error $(\omega_2)_v$ is zero for all nodes $v$ in the network,
although it is worth reiterating that this is made possible by a choice of lifting function
that relies both on prior knowledge of the nominal network states $\mathbb{S}$
and on appropriate control restrictions to constrain $F_v$.
In any case,
with this choice of $\gamma_v$,
Corollary \ref{cor:optimal-wtfv} states that the thresholding error will be completely eliminated if
$\wtfv$ is chosen to be $\tau_v \circ f_v \circ \gamma_v$.
While this is generally difficult to arrange,
we conclude by demonstrating that it can be done for a common class of well-behaved systems.

% --------------------------------------------------------------------------------------------------

\subsection{Systems with affine dynamics}
\label{sec-appl-boolean-control_subsec-affine}

We now focus on real spaces,
so that $C_v := \{c_0, c_1\} \subset \mathbb{R}$
and $S_v := \mathbb{R}$ for each piece of equipment $v$.
It is conventional in control theory to use vectors
$\vecx \in \prod_{v \in V} S_v$ and
$\vecu \in \prod_{v \in V} C_v$ to represent the state and control variables,
respectively,
in a system.
Similarly,
vectors $\wtvecx$ and $\wtvecu$ can be used to represent Boolean state and control variables.
For simplicity, we will assume that all states and controls are feasible.
This condition can be relaxed, but it greatly complicates the discussion.
When $\vecx \in \mathbb{S}$,
so that its component states are nominal (and hence feasible),
Proposition \ref{prop:net_state_sheaf} shows that we can consider $(\vecx, \vecu)$ to be a global section on $\shf{N}$.
More generally,
observe that any pair $(\vecx, \vecu)$ with feasible $\vecx$ can be interpreted
as a section of any of the sheaves that we have considered thus far.%
\footnote{Actually, for sheaves whose stalks do not use the control states, specification of just $\vecx$ is sufficient.}

Armed with this new notation,
we turn our attention to systems whose dynamics are defined by an affine function of the state and control variables.
Such equations arise in the study of linear quadratic regulator (LQR) control systems.
LQR dynamics are used in a variety of applications \cite{keafar18,com17},
and are well-studied due to their considerable practical importance.
While we use a network-based interpretation of these dynamics,
LQR models are useful in many other contexts.
These include various canonical examples in optimal control theory,
such as Kalman filtering and vehicle motion control,
that are well documented,
e.g., \cite{clarke2013FunctionalAnalysis}.

Define the system dynamics function as
\begin{equation} \label{eqn:LQR}
f(\vecx, \vecu) := A \vecx + B \vecu + \vech,
\end{equation}
where $A$ is a real matrix for which $A_{wv} = 0$ if $(w,v) \notin E$,
$B$ is a real diagonal matrix,
and $\vech$ is a real vector.
Recall from Section \ref{sec-sheaf-encoding_subsec-evolving-dynamics}
that $f$ has already been defined as a sheaf morphism from $\shf{N}$ to $\shf{L}$.
While admittedly a minor abuse of notation,
the choice of $f$ to denote the dynamics function in \eqref{eqn:LQR} is no accident.
Restricting our attention to the component of $f(\vecx, \vecu)$ corresponding to node $v$,
we see that
\begin{equation} \label{eqn:LQR_v}
[f(\vecx, \vecu)]_v
	= \sum_{w \in V} A_{wv} \vecx_w + \sum_{w \in V} B_{wv} \vecu_w + \vech_v
	= \sum_{w \in U_v} A_{wv} s_w + B_{vv} c_v + \vech_v,
\end{equation}
since $A_{wv} = 0$ for $w \notin U_v$ and $B$ is diagonal.
Thus $[f(\vecx, \vecu)]_v$ can be seen as a function that takes as input an element
$(c_v, s_v, s_{w_1}, \dotsc) \in R_v$
and produces an element of $\mathbb{R} = S_v$,
which is compatible with the definition of the dynamics function $f_v$ required by our framework.
Recall that $f_v$ is the component map for the sheaf morphism $f$ on the 1-hop neighborhood $U_v$,
and that the morphism $f$ describes how to transform an assignment for $\shf{N}$ into one for $\shf{L}$.
This suggests that if we interpret $(\vecx, \vecu)$ as a section $r$ of $\shf{N}$,
then we can also interpret $f(\vecx, \vecu)$ as a representation of the section $f(r)$ of $\shf{L}$.
This vector-based representation of sheaf morphisms offers a convenient way to simultaneously compute
the actions of all of the component maps on a given section.

As noted,
Corollary \ref{cor:optimal-wtfv} provides a recipe for eliminating thresholding error.
Proposition \ref{prop:boolean_left_inverse} demonstrated that the first requirement,
choosing $\gamma_v$ such that $\gamma_v \circ \sigma_v|_{F_v} = \id_{F_v}$,
is satisfied for the Boolean state control problem.
Our goal in this section is to demonstrate that the second requirement,
namely choosing $\wtfv = \tau_v \circ f_v \circ \gamma_v$,
is possible and reasonable when $f_v = [f(\vecx, \vecu)]_v$ \eqref{eqn:LQR_v}.

Assume that $v$ is operational if its state value is greater than some \emph{failure threshold value} $\eta_v$,
and that $v$ is failed otherwise;
mathematically,
$\Omega_v := [\eta_v, \infty)$ and $\Phi_v := (-\infty, \eta_v)$.
Note that the concept of the threshold value is distinct from that of the nominal states,
which represent best-case state behavior.
It is allowable to choose $\eta_v := (s_\Omega)_v$,
but this is not required,
and is typically not the most sensible choice.
Returning to the lighting example,
where the node state values represent voltages,
recall that we identified
$(s_\Phi)_v = 0\text{V}$ and $(s_\Omega)_v = 120\text{V}$
as the incandescent light fixture's nominal states.
We may determine by experimentation that the fixture produces enough light to be considered operational
when its voltage is above $\eta_v = 40\text{V}$,
but below that voltage,
the fixture produces too little light to be useful.

Referring to the constructions in Section \ref{sec-appl-boolean-control_subsec-boolean-discretization},
we explicitly define the thresholding functions as
$\chi_v(c_v) := (c_v - c_0)/(c_1-c_0)$
and
$\tau_v(s_v) := H(s_v - \eta_v)$,
where $H$ is the Heaviside function.
For the components of the lifting function $\gamma_v$ \eqref{eqn:appl-gamma_v},
we have $\chi_v^{-1}(\wtcv) = (c_1 - c_0) \wtcv + c_0$
and let $\rho_v(\wtsv) = ((s_\Omega)_v - (s_\Phi)_v) \wtsv + (s_\Phi)_v$.
Therefore,
the Boolean dynamics function $\wtfv$ that eliminates approximation error for the Boolean state control problem,
$\wtfv = \tau_v \circ f_v \circ \gamma_v$,
is explicitly given by
\begin{align} \label{eqn:appl_optimal-wtfv}
\wtfv(\wtcv, \wtsv, \widetilde{s_{w_1}}, \dotsc)
	&= H \left( \sum_{w \in U_v} A_{wv} \rho_w(\wtsw) + B_{vv} \chi^{-1}_v(\wtcv) + \vech_v - \eta_v \right) \\ \nonumber
	&= H \left( \sum_{w \in U_v} A_{wv} \left[ ((s_\Omega)_w - (s_\Phi)_w) \wtsw + (s_\Phi)_w \right] \right. \\ \nonumber
	& \left. \phantom{= H \left( \sum_{w \in U_v} \right.} + B_{vv} \left[ (c_1 - c_0) \wtcv + c_0 \right] + \vech_v - \eta_v \right).
\end{align}

Just as the action of the sheaf morphism $f$ can be computed by applying the system dynamics function $f(\vecx, \vecu)$,
we can define a discretized system dynamics function $\widetilde{f}(\wtvecx, \wtvecu)$
that represents the action of the sheaf morphism $\widetilde{f}$.
First,
define
\[
\chi^{-1}(\wtvecu) := D_c \wtvecu + c_0 \vecone \qquad \text{and} \qquad \rho(\wtvecx) := D_s \wtvecx + \vechs,
\]
where $D_c = \diag{(c_1 - c_0)}$,
$\vecone$ is an all-ones vector,
$D_s$ is a diagonal matrix with $(D_s)_{vv} = ((s_\Omega)_v - (s_\Phi)_v)$,
and $\vechs$ is a vector with $(\vechs)_v = (s_\Phi)_v$.
Observe that $\left[ \chi^{-1}(\wtvecu) \right]_v = \chi^{-1}_v(\wtcv)$ and $\left[ \rho(\wtvecx) \right]_v = \rho_v(\wtsv)$.
With these functions,
we can write
\[
\Gamma(\wtvecx, \wtvecu) = (\rho(\wtvecx), \chi^{-1}(\wtvecu)),
\]
which captures the aggregate action of the sheaf morphism with the same symbol.%
\footnote{Note, however, that $\left[ \Gamma(\wtvecx, \wtvecu) \right]_v = (\rho_v(\wtsv), \chi^{-1}_v(\wtcv)) \neq \gamma_v(\wtcv, \wtsv, \widetilde{s_{w_1}}, \dotsc)$.}
The sheaf morphism $T$ can similarly be represented as $T(\vecx) = \mathbf{H}(\vecx - \veceta)$,
where $\mathbf{H}$ applies the Heaviside function component-wise
and $\veceta$ is a vector with $\veceta_v = \eta_v$.
We can then define
\begin{equation} \label{eqn:appl_wtf-morphism}
\widetilde{f}(\wtvecx, \wtvecu) := (T \circ f \circ \Gamma)(\wtvecx, \wtvecu) = \mathbf{H} \left( M_1 \wtvecx + M_2 \wtvecu + \vecy \right),
\end{equation}
where
$M_1 := A D_s$,
$M_2 := B D_c$,
and $\vecy := A \vech_s + c_0 B \mathbbm{1} + \vech - \veceta$.
A straightforward computation shows that
$\left[ \widetilde{f}(\wtvecx, \wtvecu) \right]_v = \wtfv(\wtcv, \wtsv, \widetilde{s_{w_1}}, \dotsc)$ \eqref{eqn:appl_optimal-wtfv}.

We emphasize that Proposition \ref{prop:boolean_left_inverse}
and the choice of $\wtfv = \tau_v \circ f_v \circ \gamma_v$
together ensure that no error is incurred by the thresholding process for the
Boolean state control problem with an affine dynamics function \eqref{eqn:LQR}.
This implies that $\epsilon = 0$ in Theorem \ref{sec-btm-relaxation_subsec-lifting-thresh-thm},
and hence Corollary \ref{cor:gen_static_sheafrelaxed_optimal} applies in this situation.
Notice that both the individual discrete dynamics function $\wtfv$ \eqref{eqn:appl_optimal-wtfv}
and the discretized system-level function $\widetilde{f}(\wtvecx, \wtvecu)$ \eqref{eqn:appl_wtf-morphism}
consist of a Heaviside function applied to an affine function of the input arguments.
This is desirable because such functions appear as relaxations in a variety of optimization contexts
\cite{farkea17,easkle10,boyvan04},
and modern integer programming toolboxes provide a means to effectively solve problems involving such functions.

%% file: conclusion.tex
% !TeX root=./control-master.tex

\section{Conclusions}
\label{sec-conclusion}

By providing an explicit bound on the amount of error in solving a discretized optimization problem as an approximation to an optimization problem on a network,
our results provide a means to understand and evaluate critical aspects of discrete relaxations in a systematic way.
The bound therefore supports the use of efficient discretized optimization solvers
for continuous network optimization problems.

Our sheaf-theoretic result is general, yet it is easy to apply.
Given that every network optimization problem can be encoded in our sheaf-theoretic framework,
sheaves are always present in the treatment of network design and control.
While the bound that we have constructed can be loose in practice,
we have demonstrated that it can be tightened by incorporating problem-specific information.
Our framework thus has broad theoretical applicability,
and is a useful practical guide for the construction of approximate discretizations for efficiently finding solutions.

An interesting open problem that naturally stems from the constructions in this work
is the classification of a \emph{minimal} feasible set $F_v$ some vertex $v$.
Every element of such an $F_v$ would correspond to some physically-realizable configuration of the system,
although restrictions on control actions may prevent a user from actually realizing each of these states.
Note that since $F_v$ must naturally take the system dynamics into account (perhaps indirectly),
a full description of a minimal $F_v$ will almost certainly be problem-specific.